\documentclass[a4paper,11pt,twoside]{article}
\usepackage{amssymb,amsmath,latexsym,geometry,fancyhdr,lineno,hyperref,titletoc}
\contentsmargin{0pt}

\dottedcontents{section}[2em]{\vspace{-1mm}\small}{2em}{0pt}
\dottedcontents{subsection}[5em]{\vspace{-1mm}\small}{3em}{5pt}

\hypersetup{colorlinks,linkcolor= blue,citecolor=blue}
\newtheorem{theorem}{Theorem}[section]

\newtheorem{lemma}[theorem]{Lemma}
\newtheorem{proposition}[theorem]{Proposition}
\newtheorem{remark}[theorem]{Remark}

\numberwithin{equation}{section}
\allowdisplaybreaks[0]

\begin{document}
\title{{\bf\Large  Nonradial normalized solutions for a quasilinear Schr\"{o}dinger equation via dual approach}}
\author{\\
{ \textbf{\normalsize Lin Zhang}}\\
{\it\small College of Mathematics and Statistics,}\\
{\it\small Xinyang Normal University,}\\
{\it\small Henan 464000, P.R.China}\\
{\it\small e-mail address:
LinZhangYH@163.com}
}
\date{}
\maketitle
{\bf\normalsize Abstract.} {\small
In this paper, we will utilize the dual method to construct multiple nonradial normalized solutions of the following quasilinear Schr\"{o}dinger equation:
\begin{equation*}
  -\Delta u-\Delta(|u|^{2})u-\mu u=|u|^{p-2}u, \qquad in \quad \mathbb{R}^N,
\end{equation*}
subject to a mass-subcritical constraint. It should be emphasized that the nonradial result of this equation is new. Besides that, when considering the nonradial problem, it is necessary to construct a new workspace to ensure the compactness. Meanwhile, in this paper, we will expand on the method mentioned in \cite{ZLW22TMNA}.}




\pagestyle{fancy}
\fancyhead{} 
\fancyfoot{} 
\renewcommand{\headrulewidth}{0pt}
\renewcommand{\footrulewidth}{0pt}
\fancyhead[CE]{ \textsc{Lin Zhang}}
\fancyhead[CO]{ \textsc{Multiple solutions to
a class of $p$-Laplacian Schr\"{o}dinger equations}}
\fancyfoot[C]{\thepage}


\newpage
\section{Introduction}

  In this paper, we investigate the multiple non-radial normalized solutions to a class of
quasi-linear Schr\"{o}dinger equations:
\begin{equation*}
 i\Psi_{t}=\Delta \Psi+\Delta(|\Psi|^{2})\Psi+|\Psi|^{p-2}\Psi,
\quad (t,x)\in \mathbb{R}^{+}\times\mathbb{R}^N.
\end{equation*}
We consider the standing wave solutions of the form: $\Psi(t,x)= e^{i\mu t}u(x), $ where $\mu<0$ , and $u(x):\mathbb{R}^N\rightarrow \mathbb{R}$ is a real value function with $||u||_{L^2}^{2}=\lambda>0$.
This leads us to mainly consider the existence and multiplicity of solutions of (\ref{qneq1.1}) under the constraint $\mathcal{X}_{\lambda}:=\{u\in H^{1}(\mathbb{R}^N)| \int_{\mathbb{R}^N}|\nabla u|^2 u^2 dx<\infty\ ,\|u\|^{2}_{2}=\lambda\}$,
\begin{equation}\label{qneq1.1}
\left\{
\begin{aligned}
& -\Delta u-\Delta(|u|^{2})u-\mu u=|u|^{p-2}u,\\
& \int_{\mathbb{R}^N}|u|^2dx =\lambda,
\end{aligned}
\right.
\end{equation}
where $\mu$ appears as a Lagrange multiplier. It should be noted that $u\in \mathcal{X}_{\lambda}$ is called a weak solution if $J'(u)\psi = \mu \int_{\mathbb{R}^N} u \psi dx $ for any $\psi \in C^\infty_0(\mathbb{R}^N)$, where
\begin{equation}\label{nfj1.2}
J(u)=\frac{1}{2}\int_{\mathbb{R}^N}|\nabla u|^2dx+\int_{\mathbb{R}^N}|\nabla u|^2 u^2 dx-\frac{1}{p}\int_{\mathbb{R}^N}|u|^p dx .
\end{equation}
and \[  J'(u)\phi:=\lim_{t\rightarrow 0^+}\frac{J(u+t\phi)-J(u)}{t}.\]

Our main result is the following theorem.
\begin{theorem}\label{Thm1.1}
\rm{  Assume  $m\in [2,\frac{N}{2}]$ is a fixed integer, $N\geq 4$, $N-2m\neq 1$ and $p\in (2,4+\frac{4}{N})$. Then the following statements hold.\\
(i). Let $p\in(2, 2+\frac{4}{N})$. For each $\lambda>0$ fixed, the equation (\ref{qneq1.1}) has infinitely many  pairs of non-radial solutions.\\
(ii). Let $p\in [2+\frac{4}{N}, 4+\frac{4}{N})$. For each $k\in \mathbb{N}$, there exists a $\lambda_{k}>0, $  such that the equation (\ref{qneq1.1}) has at least $k$ pairs of distinct non-radial solutions for all $\lambda>\lambda_{k}$.\\
(iii). For fixed $\lambda$, if $(u_{k})$ is the solution sequence of equation (\ref{qneq1.1}), then $J(u_{k})<0$, and $J(u_{k})\rightarrow 0$, as $k\rightarrow \infty$.}
\end{theorem}
\begin{remark}\label{Rek1.1}
\rm{The non-radial solutions given by Theorem \ref{Thm1.1} are sign-changing.
}
\end{remark}

In the recent forty years, there are many researches on the sign-changing solutions of nonlinear elliptic equations. In 1983, the problem of the existence of non-radial solutions for a class of semilinear elliptic equations was first given by Berestycki and Lions in \cite{2BL83ARMA}, and had been open for a long time. After ten years, in \cite{BW93JFA} Bartsch and Willem found an unbounded sequence of non-radial solutions of (\ref{seq1.3}) for $N=4$ or $N\geq 6$,
\begin{equation}\label{seq1.3}
\left\{
\begin{aligned}
& -\Delta u-b(|x|)u=g(|x|,u)\\
& u\in H^{1}(\mathbb{R}^{N})
\end{aligned}
\right.,
\end{equation}
provided $b$ and $g$ satisfy certain growth conditions and $f$ is odd in $u$. After that, the study of non-radial solutions of semilinear equations can be divided into two categories. One is consider the semilinear equations without prescribing the $L^{2}$-constraint(e.g.,\cite{ B99Nankai, B01JFA, BCW00JMZ, BW96TMNA, BW99TMNA,BW03TMNA,  CCN97JMA, CCN98EJDE, LZ00AMS,   LS01JDE,  LW08FMC, MPW12JEMS} and references therein). The another one is consider the existence of non-radial normalized solutions, see \cite{JL19No, JL20CVPDE}. It is worth mentioning that \cite{JL19No} first considers the existence of non-radial normalized solutions for mass-subcritical. Different from \cite{JL19No}, \cite{JL20CVPDE} considers the existence of non-radial normalized solutions for mass-supercritical.

Besides that, there are some studies on the the quasilinear equation without the $L^{2}$-constraint. For example, in \cite{LLW14CPDE}, Liu, Liu, and Wang had used the variational perturbation method to illustrate the multiplicity of sign-changing solutions for the following general quasilinear equations
\begin{equation}\label{geq1.4}
\int_{\Omega}\Sigma_{i.j=1}^{N}a_{ij}(x,u)D_{i}uD_{j}\phi dx+\frac{1}{2}\int_{\Omega}\Sigma_{i.j=1}^{N}D_{s}a_{ij}(x,u)D_{i}uD_{j}u\phi dx-\int_{\Omega}f(x,u)\phi dx=0,
\end{equation}
where $ \Omega \subset \mathbb{R}^{N}$ is a bounded domain with smooth boundary. In \cite{LWW04CPDE}, Liu, Wang and Wang had obtained the existence of sign-changing solutions for a class of quasilinear equations by using the Nehari method.

However, under the mass-subcritical case, there is no research on the multiplicity of non-radial normalized solutions for equation (\ref{qneq1.1}). And it should be noted that the constraint minimization method and the variational perturbation method are not applicable here. The reason is that the constraint minimization methods are commonly used to prove the existence of solutions. For the variational perturbation method, when we use the minimax principle to obtain multiple critical points of the following perturbed functional  \begin{equation}\label{pfj1.7}
J_{\nu}(u):=\frac{\nu}{4}\int_{\mathbb{R}^N}|\nabla u|^4 dx+J(u),
\end{equation}
which is defined on $W^{1,4}_0,$
and $\nu \in (0,1]$ is a parameter, the critical values we expect to obtain are accumulate at zero level from below. This makes it difficult to distinguish critical values after passing to limit with $\nu \rightarrow 0$ in the perturbation setting as done in \cite{LW14JDE, LLW13PAMS, LLW13JDE}.

Our plan in this paper is to extend the method used in \cite{ZLW22TMNA}. It should be noted that the result on the multiplicity of non-radial normalized solutions of (\ref{qneq1.1}) seems to be new. Here, we use the duality method to transform the qusilinear equation (\ref{qneq1.1}) into the semilinear problem
\begin{equation}\label{deq1.8}
\left\{
\begin{aligned}
& - \Delta v =|f|^{p-2}ff^{'}(v)+\mu ff^{'}(v),\\
& \int_{\mathbb{R}^N}|f(v)|^2dx =\lambda,
\end{aligned}
\right.
\end{equation}
where $v:=f^{-1}(u), f^{'}(t)=\frac{1}{\sqrt{1+2f^{2}(t)}}, -f(-t)=f(t).$ Then the variational functional (\ref{nfj1.2}) can be transformed into
\begin{equation}\label{dnfj1.10}
I(v)=\frac{1}{2}\int_{\mathbb{R}^N}|\nabla v|^2 dx-\frac{1}{p}\int_{\mathbb{R}^N}|f(v)|^p dx.
\end{equation}

In order to prove the multiplicity of non-radial normalized solutions to the equation (\ref{qneq1.1}), it is not only necessary to use the dual approach to transform the problem, but also to reset the workspace to ensure the compactness. Then, we will consider the multiplicity of non-radial normalized solutions on the constraint $\mathbb{Y}_{\lambda}=\mathbf{Y}_{\mathbf{G}}\cap Y_{\lambda}$, where
\begin{equation}\label{Yn1.11}
Y_\lambda=\{ v\in H^1(\mathbb{R}^N),  ||f(v)||^{2}_{2} =\lambda\},
\end{equation}
the definition of $\mathbf{Y}_{\mathbf{G}}$ will be given based on the following notations. Assume  $m\in [2,\frac{N}{2}]$ is a fixed integer, $N\geq 4$ and $N-2m\neq 1$. Fixing $\tau\in\mathfrak{O}(N)$ satisfies $\tau(x_{1}, x_{2}, x_{3})=(x_{2}, x_{1}, x_{3})$ for $x_{1}, x_{2}\in\mathbb{R}^{m}$ and $x_{3}\in\mathbb{R}^{N-2m}$, where $x=(x_{1}, x_{2}, x_{3})\in\mathbb{R}^{N}=\mathbb{R}^{m}\times\mathbb{R}^{m}\times\mathbb{R}^{N-2m}$ and $\mathfrak{O}(N):=\{\mathrm{O}_{N\times N}| ~orthogonal ~metrices\}$. Now we define
$$\mathfrak{X}_{\tau}:=\{u\in\mathcal{X}, u(\tau x)=-u(x) ~for ~all~ x\in \mathbb{R}^{N} \},$$
$$\mathfrak{Y}_{\tau}:=\{v\in H^{1}(\mathbb{R}^{N}), v(\tau x)=-v(x) ~for ~all~ x\in \mathbb{R}^{N} \},$$
$$\mathbf{G}:=\mathfrak{O}(m)\times\mathfrak{O}(m)\times\mathfrak{O}(N-2m).$$
Let $H^{1}_{\mathbf{G}}(\mathbb{R}^{N})$ denote the subspace of invariant functions with respect to $\mathbf{G}$, where $\mathbf{G}\subset\mathfrak{O}(N)$ acts isometrically on $H^{1}(\mathbb{R}^{N})$. It is clear that $0$ is the only radial function of $\mathfrak{X}_{\tau}$ and $\mathfrak{Y}_{\tau}$. For convenience, we set $\mathbf{X}_{\mathbf{G}}:=H^{1}_{\mathbf{G}}(\mathbb{R}^{N})\cap\mathfrak{X}_{\tau}$ and $\mathbf{Y}_{\mathbf{G}}:=H^{1}_{\mathbf{G}}(\mathbb{R}^{N})\cap\mathfrak{Y}_{\tau}$. It should be emphasized that the embedding $\mathbf{Y}_{\mathbf{G}}\subset L^{p}(\mathbb{R}^{N})$ is compact, see \cite{L82JFA, W96B}. The reason for defining $\mathfrak{Y}_{\tau}$ is to serve the dual functional (\ref{dnfj1.10}).  We also point out that due to this change of variable, we would lose the $L^2$-mass invariance of the problem, which in mathematical terms is a certain scaling invariance of the problem and was used heavily in the arguments for the existence of normalized solutions in the classical works \cite{C03NYU, CL82CMP, CLW12EJAM, CJS10NA}.
Our work will fill in this part to have the new constraint about $f(v)$ worked out for the theory of multiplicity.

The paper is organized as follows.
In Section 2, first we will introduce the framework of the dual method for normalized solutions by showing and using some properties of the dual function $f$. Lemma 2.2 will give the equivalence relation between (\ref{nfj1.2}) and (\ref{dnfj1.10}), that is, $I(v)=J(u)$ under the constraint $\mathbb{Y}_{\lambda}:=\mathbf{Y}_{\mathbf{G}}\cap Y_{\lambda}$ and $\mathbb{X}_{\lambda}:=\mathbf{X}_{\mathbf{G}}\cap\mathcal{X}_{\lambda}$ respectively.

In Section 3, we will use the minimax principle and the deformation argument to establish the multiplicity of non-radial solutions. It is worth mentioning that due to the dual function $f$ being nonlinear, when the functional (\ref{dnfj1.10}) is considered under the constraint $\mathbb{Y}_{\lambda}$, we need to overcome two difficulties. The one is that the stretching method for quasi-linear  functional (\ref{nfj1.2}) is not applicable here. The another one is that the explanation of the existence of Lagrange multipliers mentioned in [\cite{2BL83ARMA}, Lemma 3] is no longer applicable here. Instead we will use the approximation (Lemma \ref{Lem3.3}) and the properties of genus to prove the Theorem \ref{Thm1.1}.
For a matter of convenience,  we will use $\int$ to represent $\int_{\mathbb{R}^{N}}$.

\section{The dual framework}

For completeness, we give the following lemmas. However, to avoid repetitions, the proof of the following lemmas can be referred to \cite{ZLW22TMNA}.
\begin{lemma}
\label{Lem2.1}
\rm{ $f^{'}(t)=\frac{1}{\sqrt{1+2f^{2}(t)}}, t\geq 0, f(t)=-f(-t)$
\begin{enumerate}
 \item[(1)] $f(t)$ is uniquely defined , invertible and smooth;
 \item[(2)] $\lim_{t\rightarrow 0}\frac{|f(t)|}{|t|}=1$ , $\lim_{t\rightarrow\infty}\frac{|f(t)|}{\sqrt{|t|}}=2^{\frac{1}{4}}$;
 \item[(3)] $|f^{'}(t)|\leq 1$ , $|f(t)|\leq|t|$;
 \item[(4)] $\frac{|f(t)|}{2}\leq|t|f^{'}(t)\leq|f(t)|$;
 \item[(5)] $|f(t)|\leq 2^{\frac{1}{4}}|t|^{\frac{1}{2}}$ , for all $t\in\mathbb{R}$;
 \item[(6)] $there~~ exist ~~C_1, C_2>0 $ , such that
 $|f(t)|\geq C_1|t| $, if $|t|\leq 1$; $|f(t)|\geq C_2|t|^{\frac{1}{2}}$, if $|t|\geq 1$.
\end{enumerate}
}
\end{lemma}
\begin{lemma}\label{Lem2.2}
\rm{
Let $u=f(v)$, then
\begin{enumerate}
  \item[(i):] $Y:=\{u=f(v)\,| v\in H^{1}(\mathbb{R}^{N})\}= {\mathcal X} =\{u\in H^{1}(\mathbb{R}^N)| \int_{\mathbb{R}^N}|\nabla u|^2 u^2 dx<\infty\};$
  \item[(ii):] $\mathfrak{X}_{\tau}=f(\mathfrak{Y}_{\tau});$
  \item[(iii):] for each $v\in H^{1}(\mathbb{R}^{N}), \bar{J}(u)=\bar{I}(v)$, where
  \begin{equation*}
\bar{J}(u)=\frac{1}{2}\int_{\mathbb{R}^N}|\nabla u|^2 dx+\int_{\mathbb{R}^N}|\nabla u|^2 u^2 -\frac{\mu}{2}|u|^2 dx-\frac{1}{p}\int_{\mathbb{R}^N}|u|^p dx ,
\end{equation*}
\begin{equation*}
\bar{I}(v)=\frac{1}{2}\int_{\mathbb{R}^N}|\nabla v|^2 dx-\mu|f(v)|^2 dx-\frac{1}{p}\int_{\mathbb{R}^N}|f(v)|^p dx.
\end{equation*}
\end{enumerate}}
\end{lemma}

\begin{proof} The proofs of $(i)$ and $(iii)$ can see \cite{AW12NA, ZLW22TMNA}.
For $(ii)$, due to $f(v)=u$, then for each $v\in\mathfrak{Y}_{\tau}$, $u\in\mathfrak{X}_{\tau}$, we can have $-f(v)=f(v(\tau x))=u(\tau x)=-u(x)$. Then according to the proof of $(i)$, easy to get $\mathfrak{X}_{\tau}=f(\mathfrak{Y}_{\tau})$.
\end{proof}
\begin{remark}\label{Rek2.1}
\rm{
According to this lemma, it is easy to see that $I(v)=J(u)$ under the constraint $\mathbb{Y}_{\lambda}:=\mathbf{Y}_{\mathbf{G}}\cap Y_{\lambda}$ and $\mathbb{X}_{\lambda}:=\mathbf{X}_{\mathbf{G}}\cap\mathcal{X}_{\lambda}$ respectively.}
\end{remark}
\begin{lemma}\label{Lem2.3}
\rm{For $p\in (2, 4+\frac{4}{N})$, define
$$I_{\lambda}:=\inf\{I(v)| v\in \mathbb{Y}_{\lambda}\} , \lambda>0. $$
\begin{enumerate}
  \item[(i):] $I_{\lambda}$ is bounded from below, that is $I_{\lambda}> -\infty$;
  \item[(ii):] Assume $(v_{n})$ is the minimizing sequence of $I(v)$ under the constraint $\mathbb{Y}_{\lambda}$, then $(v_n)$ is bounded in $ H^{1}(\mathbb{R}^{N})$.
\end{enumerate}}
\end{lemma}
\begin{proof}
For (a), by the Gagliardo-Nirenberg inequality, there exists some $C>0$ depend only on $N$, such that for any $v\in \mathbb{Y}_{\lambda}$,
\begin{align*}
\int_{\mathbb{R}^{N}}|f(v)|^{p} dx & \leq\Big(\int_{\mathbb{R}^{N}}|f(v)|^2dx\Big)^{1-\theta}\Big(\int_{\mathbb{R}^{N}}|f(v)|^{\frac{4N}{N-2}}dx\Big)^{\theta} \\
& \leq C\lambda^{1-\theta}\Big(\int_{\mathbb{R}^{N}}|\nabla u|^2 u^2dx\Big)^{\frac{\theta N}{N-2}}\\
& =C\lambda^{\frac{4N-(N-2)p}{2(N+2)}}\Big(\int_{\mathbb{R}^{N}}|\nabla u|^{2}u^{2}dx\Big)^{\frac{(p-2)N}{2(N+2)}}<\infty,
\end{align*}
where \,$\theta=\frac{(p-2)(N-2)}{2(N+2)}$, and \,$\frac{\theta N}{N-2}<1$, when \,$p<4+\frac{4}{N}$. Thus $I(v)>-\infty$, \,$(i)$\, is proved.

For (b), according to the calculation in \,$(i)$\, and \,$|\nabla v|^{2}=|\nabla u|^{2}+2|\nabla u|^{2}u^{2}$\,
\begin{align*}
\int_{\mathbb{R}^{N}}|f(v)|^{p} dx
&\leq C\lambda^{\frac{4N-(N-2)p}{2(N+2)}}\Big(\int_{\mathbb{R}^{N}}(2|\nabla u|^{2}+1)u^{2}dx\Big)^{\frac{(p-2)N}{2(N+2)}}\\
&=C\lambda^{\frac{4N-(N-2)p}{2(N+2)}}\Big(\int_{\mathbb{R}^{N}}|\nabla v|^{2}dx\Big)^{\frac{(p-2)N}{2(N+2)}},
\end{align*}
then $$I(v)\geq\frac{1}{2}\int_{\mathbb{R}^{N}}|\nabla v|^{2}dx-C\lambda^{\frac{4N-(N-2)p}{2(N+2)}}\Big(\int_{\mathbb{R}^{N}}|\nabla v|^{2}dx\Big)^{\frac{(p-2)N}{2(N+2)}}.$$
Under the assumptions, we have \,$\|\nabla v_{n}\|^{2}_{2}$\, is bounded. Thus, \,$f(v_{n})$\, is bounded in \,$H^{1}(\mathbb{R}^{N})$\,. By (6) of Lemma \ref{Lem2.1} and the above estimate, we have
\begin{align*}
\|v_{n}\|^{2}_{2}=\int_{\mathbb{R}^{N}}|v_{n}|^{2}dx
&=\int_{|v_{n}|\leq1}|v_{n}|^{2}dx+\int_{|v_{n}|\geq1}|v_{n}|^{2}dx\\
&\leq C_{1}\lambda+\int_{\mathbb{R}^{N}}|f(v_{n})|^{4}dx\\
&\leq C_{1}\lambda+C(\int_{\mathbb{R}^{N}}|f(v_{n})|^{2}dx)^{\frac{4}{N+2}}
(\int_{\mathbb{R}^{N}}|\nabla v_{n}|^{2}dx)^{\frac{N}{N+2}},
\end{align*}
therefore, $(v_n)$ is bounded in $ H^{1}(\mathbb{R}^{N})$.
\end{proof}


\section{The multiplicity of non-radial normalized solutions}

This section is mainly dedicated to proving Theorem \ref{Thm1.1}. In order to prove the multiplicity of non-radial normalized solutions, let us recall the constraint
\begin{equation*}
\mathbb{Y}_{\lambda}=\mathbf{Y}_{\mathbf{G}}\cap Y_{\lambda}=\{v\in \mathbf{Y}_{\mathbf{G}}, \|f(v)\|^{2}_2 =\lambda \},
\end{equation*}
Now we define the minimax values for $\lambda>0$ and positive integer $k$,
$$\beta_{\lambda, k}:= \inf_{E\in \Gamma_{\lambda, k}}\sup_{v\in E} I(v), $$
where
$$\Gamma_{\lambda, k}:= \{E\in \Sigma(\mathbb{Y}_{\lambda})| \gamma(E) \geq k\},$$
$$\Sigma(\mathbb{Y}_{\lambda})=\{E\in \mathbb{Y}_{\lambda} \;|\; E \;\mbox{is closed and}\; E=-E\}$$ and $\gamma (\cdot)$ is the genus (see \cite{R65CBMS}).

Before proving Theorem \ref{Thm1.1}, we need to prove the multiplicity of non-radial normalized solutions of equation (\ref{deq1.8}) by verifying the following theorem.
\begin{theorem}\label{THm3.1}\rm{  Assume  $m\in [2,\frac{N}{2}]$ is a fixed integer, $N\geq 4$, $N-2m\neq 1$ and $p\in (2,4+\frac{4}{N})$. Then the following statements hold.
\begin{description}
  \item[(i)] for $p\in(2, 2+\frac{4}{N})$, (\ref{deq1.8}) has infinitely many pairs of non-radial solutions under the constraint $\mathbb{Y}_{\lambda}$, for all $\lambda>0$;
  \item[(ii)] for $p\in [2+\frac{4}{N}, 4+\frac{4}{N})$, and  each $k\in \mathbb{N}$, there exist a $\lambda_{k}>0, $  such that (\ref{deq1.8}) has at least $k$ pairs of distinct non-radial solutions under the constraint $\mathbb{Y}_{\lambda}$, for all $\lambda>\lambda_{k}$;
  \item[(iii)] for fixed $\lambda$, if $(v_{k})$ is the critical points sequence of $\beta_{\lambda, k}$, $k\in \mathbb{N}$, then $I(v_{k})=\beta_{\lambda, k}<0$, and $I(v_{k})\rightarrow 0$, as $k\rightarrow \infty$.
\end{description}
}
\end{theorem}
Next, we mainly use the minimax principle and the deformation argument to illustrate the multiplicity. Before that, in order to prove the compactness(Lemma \ref{Lem3.2}), we need the following Lemma \ref{Lem3.1} and Proposition \ref{Pro3.1}.

\begin{lemma}\label{Lem3.1}\rm{
Let $X$ be a Hilbert space, suppose that $G:X\rightarrow \mathbb{R}$ is a $C^{1}$ mapping. Let $\mathcal{M}=\{w\in X|G(w)=0\}$ and $G^{'}(w)w\neq 0$, for all $w\in\mathcal{M}$. Then
\begin{equation}\label{M3.1}
T_{w_{0}}\mathcal{M}=ker ~G^{'}(w_{0}),
\end{equation}
for each $w_{0}\in \mathcal{M}$, where $T_{w_{0}}\mathcal{M}$ is the tangent space of $\mathcal{M}$ at $w_{0}$, and the definition is as follows:
$$T_{w_{0}}\mathcal{M}:=\{h\in X| \exists \epsilon>0, \exists \zeta\in C^{1}((-\epsilon, \epsilon),X) ~s.t.~ w_{0}+\zeta(t)\in\mathcal{M}, \zeta(0)=\theta, \dot{\zeta}(0)=h \}.$$
}
\end{lemma}
\begin{proof}
First we show $T_{w_{0}}\mathcal{M}\subset ker ~G^{'}(w_{0})$. Since $G(w_{0}+\zeta(t))=0$ for any $h\in T_{w_{0}}\mathcal{M}$, then we have
$$G^{'}(w_{0})h=\frac{d}{dt}G(w_{0}+\zeta(t))=0.$$ Thus $T_{w_{0}}\mathcal{M}\subset ker ~G^{'}(w_{0})$.\\
To show $ker ~G^{'}(w_{0})\subset T_{w_{0}}\mathcal{M}$, we claim that if $h\in ker ~G^{'}(w_{0})$, then there exists $\eta\in C^{1}((-\epsilon, \epsilon),X_{1})$, $\eta(0)=\theta$ satisfies the following equation:
\begin{equation}\label{M1}
G(w_{0}+th+\eta(t))=0,
\end{equation}
where $X_{1}$ is the complement of $ker ~G^{'}(w_{0})$ and $\epsilon>0$ is small enough.\\
Since $G(w_{0})=0$ and $G^{'}(w_{0}): X_{1}\rightarrow \mathcal{R}$ is an invertible mapping, then by Implicit Function Theorem(for short IFT) we can have $G(w)=0$ in $U$, where $U$ is the $\epsilon$-neighborhood of $w_{0}$ satisfies $w\in U$,  $\|w-w_{0}\|_{X}\leq \epsilon$. So we can set $w=w_{0}+th+\eta(t)$, $t\in(-\epsilon, \epsilon)$, that is, the claim is proved. \\
Now setting $\zeta(t)=th+\eta{t}$, we have $\zeta(0)=\theta$, $\dot{\zeta}(0)=h+\dot{\eta}(0)$. Then from (\ref{M1}), it follows that $G(w_{0}+\zeta(t))=0$ and $G^{'}(w_{0})(h+\dot{\eta}(0))=0$. Thus $\dot{\eta}(0)\in ker ~G^{'}(w_{0})$. However, due to $\dot{\eta}(0)\in X_{1}$, this implies that $\dot{\eta}(0)=\theta$, i.e. $\dot{\zeta}(0)=h$. Therefore, for any $h\in ker ~G^{'}(w_{0})$, we can obtain that $G^{'}(w_{0})h=0$ and $w_{0}+\zeta(t)\in\mathcal{M}$, $\zeta{0}=\theta$, $ \dot{\zeta}(0)=h$, that is, $h\in T_{w_{0}}\mathcal{M}$.
\end{proof}
\begin{remark}\label{Rek3.1}\rm{
This lemma will be used to explain the existence of Lagrange multiplies in the proof of Lemma \ref{Lem3.2} below. It should be point out that due to the dual function $f$ is nonlinear, the method of explaining the existence of Lagrange multipliers in [\cite{2BL83ARMA}, Lemma 3.] is not applicable. In other words, the constrained manifold $\mathbb{Y}_{\lambda}$ in this paper cannot be regarded as a standard sphere in $L^{2}(\mathbb{R}^{N})$. And the above lemma can refer from \cite{C05SVB}.
}
\end{remark}

\begin{proposition}\label{Pro3.1}\rm{
If $(v_{n})$ is bounded in $\mathbb{Y}_{\lambda}$, $\lim_{n\rightarrow \infty}\mu_{n}=\mu<0$ and
\begin{equation}\label{P3.1}
I^{'}(v_n)-\mu_{n}ff^{'}(v_n) \rightarrow I^{'}(v)-\mu ff^{'}(v)=0.
\end{equation}
Then we can have  $\int|v_{n}|^{2}dx\rightarrow \int|v|^{2}dx$.
}
\end{proposition}
\begin{proof}
If $(v_{n})$ is bounded in $\mathbb{Y}_{\lambda}$, then we can have $v_{n}\rightharpoonup v$ in $\mathbb{Y}_{\lambda}$. Since $\|f(v_n)\|_{22^{*}}<\infty$, $2<p<22^*$, it is easy to know that $[f(v_{n})]^{p}\rightarrow[f(v)]^{p} $ a.e. in $\mathbb{R}^{N}$. Then following from [\cite{1BL83ARMA}, Theorem A.I], we can have
\begin{align*}\int|f(v_{n})|^{p}dx\rightarrow \int|f(v)|^{p}dx.\end{align*}
According to $\lim_{n\rightarrow \infty}\mu_{n}=\mu<0$ and (\ref{P3.1}), multiplying
 $I^{'}(v_{n})-\mu_{n}f(v_{n})f^{'}(v_{n})$ by $\frac{f(v_{n})}{f^{'}(v_{n})}$, and using $f^{''}(v_{n})=-2f(v_{n})f^{'}(v_{n})^{4}$, we can have \begin{align*}\int|\nabla v_{n}|^{2}dx+2\int|\nabla v_{n}|^{2}f^{2}f^{'}(v_{n})^{2}dx-\mu\int|f(v_{n})|^{2}dx-\int|f(v_{n})|^{p}dx\\
\rightarrow
\int|\nabla v|^{2}dx+2\int|\nabla v|^{2}f^{2}f^{'}(v)^{2}dx-\mu\int|f(v)|^{2}dx-\int|f(v)|^{p}dx,
\end{align*}
i.e. $\int|\nabla v_{n}|^{2}dx\rightarrow\int|\nabla v|^{2}dx$ , $\int|f(v_{n})|^{2}dx\rightarrow\int|f(v)|^{2}dx$.\\
By the Lagrange theorem, we can have
\begin{align*}
f(v_{n})-f(v)=f^{'}(v+\theta(v_{n}-v))(v_{n}-v),
\end{align*}
where $\theta=\theta(x)\in(0,1)$, and $f^{'}(v+\theta(v_{n}-v))\geq\sqrt{\delta}>0 $ for some $\delta>0$.
Then following from the Br$\acute{e}$zis-Lieb Lemma, we can have
\begin{align*}
0\leftarrow\|f(v_{n})-f(v)\|^{2}_{2}=\|f^{'}(v+\theta(v_{n}-v))(v_{n}-v)\|^{2}_{2}\geq
\delta\|v_{n}-v\|^{2}_{2},
\end{align*}
i.e. $\int|v_{n}|^{2}dx\rightarrow \int|v|^{2}dx$, as $n\rightarrow \infty$.
\end{proof}

From Lemma \ref{Lem3.1}, Proposition \ref{Pro3.1}, and the fact that the embedding $\mathbf{Y}_{\mathbf{G}}\subset L^{p}(\mathbb{R}^{N})$ is compact, see \cite{L82JFA, W96B}. We can have
\begin{lemma}\label{Lem3.2}
\rm{
For $p\in(2, 4+\frac{4}{N})$, $m\in [2,\frac{N}{2}]$ is a fixed integer, $N\geq 4$, $N-2m\neq 1$. The functional $I_{|\mathbb{Y}_{\lambda}}$ satisfies the $(PS)_{c}$ for all $c<0$.}
\end{lemma}
\begin{remark}\label{Rek3.2}\rm{ The proof of the above lemma can refer to \cite{ZLW22TMNA}. It should be note that in the above lemma, $c<0$ is necessary. If not, we cannot deduce $\mu<0$, let alone $v\in\mathbb{Y}_{\lambda}$. Moreover, due to $c<0$, it is easy to know that $v\neq 0$.}
\end{remark}
Next, to prepare the application of the minimax principle, we need the following lemma.
\begin{lemma}\label{Lem3.3}
\rm{Let $N\geq 4$, $ N-2m\neq 1$ and $\beta_{\lambda, k}$ defined above. For each positive integer $k$,
\begin{description}
  \item[(i).] If $p\in(2, 2+\frac{4}{N})$, $\beta_{\lambda, k}<0$, for any $\lambda>0$.
  \item[(ii).] If $p\in[2+\frac{4}{N},4+\frac{4}{N})$, there exists $\lambda_{k}>0$  such that
$\beta_{\lambda,k}<0$, for any $\lambda\in(\lambda_{k},\infty)$.
\end{description}}
\end{lemma}
\begin{proof}
Since $C^{\infty}_{0}(\mathbb{R}^{N})$ is dense in $L^{2}(\mathbb{R}^{N})$, then there exists $M_{k} \subset L^{2}(\mathbb{R}^{N})$, where $M_{k}$ denotes a $k$-dimensional subspace of $L^{2}(\mathbb{R}^{N})$, and consists of smooth functions with compact supports. Then similar to the proof of Lemma 3.4 in \cite{ZLW22TMNA}, it is easy to prove this lemma.
\end{proof}
\begin{remark}\label{Rek3.3}\rm{
The idea in the above lemma is applicable to the whole $H^{1}(\mathbb{R}^{N})$. In other words, this idea applies not only to the invariant function space $\mathbb{Y}_{\lambda}$ here, but also to the radial function space $H_{r}^{1}(\mathbb{R}^{N})$. The reason is that any function in $H^{1}(\mathbb{R}^{N})$ can be approximated by the function in $C_{0}^{\infty}(\mathbb{R}^{N})$.
}
\end{remark}
In order to prove the multiplicity by the minimax principle, it is inevitable to use the properties of genus and the deformation lemma (see Lemma \ref{Lem3.4} and lemma \ref{Lem3.5} respectively). First, here we recall some definition and notations in \cite{R65CBMS}:
$\Sigma(\mathbb{Y}_{\lambda})$ denotes the set of closed and symmetric (respect to the original) subsets of $\mathbb{Y}_{\lambda}$; $\gamma(E)$ is defined as the least integer $n$, such that there exists an odd mapping $\psi\in C(E, \mathbb{R}/\{0\})$; $K_{c}:=\{v\in\mathbb{Y}_{\lambda}\big|I(v)=c, I^{'}_{|\mathbb{Y}_{\lambda}}(v)=0\}$; $N_{\delta}(E):=\{v\in\mathbb{Y}_{\lambda}, \|v-E\|_{H^{1}}<\delta\}$;
$\tilde{\mathbb{Y}}_{\lambda}:=\{v\in\mathbb{Y}_{\lambda}, I^{'}_{|\mathbb{Y}_{\lambda}}(v)\neq 0\}$; $I^{c}_{|\mathbb{Y}_{\lambda}}:=\{v\in\mathbb{Y}_{\lambda}, ~I(v)\leq c\}$.
\begin{lemma}\label{Lem3.4}\rm{
For $E,~F\in\Sigma(\mathbb{Y}_{\lambda})$, the following statements hold.\noindent
\begin{enumerate}
  \item[(i):] if $\gamma(E)\geq 2$, then there exist infinite points in E;
\item[(ii):]  for $u\neq 0$, $\gamma(\{u\}\cup\{-u\})=1$;
\item[(iii):]  for $E,~ F\in\Sigma(\mathbb{Y}_{\lambda})$, if exist odd  mapping $\psi\in~ C(E,F)$, then $\gamma(E)\leq\gamma(F)$;
\item[(iv):]  if $E\subset F$, $\gamma(E)\leq\gamma(F)$;
\item[(v):]   $\gamma(E\cup F)\leq\gamma(E)+\gamma(F)$;
\item[(vi):]   if E is compact, then $\gamma(E)<\infty$ and there exists $\delta>0$, such that $ N_{\delta}(E)\in\mathbb{Y}_{\lambda}$, $\gamma(N_{\delta}(E))=\gamma(E)$;
\item[(vii):] if $\gamma(F)<\infty$, $\gamma(\overline{E/F})\geq\gamma(E)-\gamma(F)$.
\end{enumerate}}
\end{lemma}
\begin{lemma}\label{Lem3.5}\rm{
Let $I\in C^{1}(\mathbb{Y}_{\lambda}, \mathbb R)$ satisfy $(PS)_{c}, ~c<0$. Then there exist $0<\epsilon <\bar{\epsilon}$, and $\eta\in C([0,1]\times\mathbb{Y}_{\lambda}, \mathbb{Y}_{\lambda})$, such that
\begin{enumerate}
\item[(1):]$\eta(0,v)=v$, for all $v\in\mathbb{Y}_{\lambda}$;
\item[(2):] $\eta(t,v)=v$, for all $t\in[0,1]$, if $v\notin I^{c+\bar{\epsilon}}_{|\mathbb{Y}_{\lambda}}/(I^{c-\bar{\epsilon}}_{|\mathbb{Y}_{\lambda}} \cup N_{\frac{\delta}{8}}(K_c))$;
\item[(3):] $\forall t\in[0,1]$, $\eta(t,v)$ is a homeomorphism of $\mathbb{Y}_{\lambda}$ to  $\mathbb{Y}_{\lambda}$;
\item[(4):]$\eta(1,I^{c+\epsilon}_{|\mathbb{Y}_{\lambda}}/N_{\delta}(K_c))\subset I^{c-\epsilon}_{|\mathbb{Y}_{\lambda}}$;
\item[(5):]if $K_{c}=\emptyset$, $\eta(1, I^{c+\epsilon}_{|\mathbb{Y}_{\lambda}})\subset I^{c-\epsilon}_{|\mathbb{Y}_{\lambda}}$;
\item[(6):]if $I(v)$ is even, $\eta(t,v)$ is odd in $v$.
\end{enumerate}}
\end{lemma}
Following from the above two lemmas, we can prove the following minimax principle. Note that in \cite{JL19No} there is another proof.
\begin{theorem}\label{Thm3.2}\rm{
Let $I\in C^{1}(\mathbb{Y}_{\lambda}, \mathbb R)$. Assume that $I_{|\mathbb{Y}_{\lambda}}$ is bounded from below and satisfy $(PS)_{c},~c<0$.
 Assume that for each $k\in N, ~\Gamma_{k}\neq\emptyset$. Denote $$c_{k}:=\inf_{E\in\Gamma_{k}}\sup_{v\in E}I(v)$$ as the minimax values. Then we have:
\begin{enumerate}
\item[\underline{(i)}:] for any $k\geq 1$, $c_{k}$ is a critical value of $I_{|\mathbb{Y}_{\lambda}}$ if  $c_{k}<0$, and $-\infty<c_{1}\leq c_{2}\leq\cdots\leq c_{k}\leq\cdots<0$;
\item[\underline{(ii)}:] if $c_{k}=c_{k+1}=\cdots=c_{k+l-1}:=c$, $k,l\geq 1$, then $\gamma(K_{c})\geq l$. Moreover, if $l\geq 2$, $I_{|\mathbb{Y}_{\lambda}}$ has infinitely critical points in $K_{c}$;
\item[\underline{(iii)}:]if $c_{k}<0$ for all $k\geq 1$, then $c_{k}\rightarrow 0^{-}$ as $k\rightarrow\infty$.
\end{enumerate}}
\end{theorem}
\begin{proof}
This theorem can be proved by referring to the Theorem 3.7 of \cite{ZLW22TMNA}.
\end{proof}

{\bf{Proof of Theorem \ref{THm3.1}.}} On the basis of the above lemmas and theorems, we can prove Theorem \ref{THm3.1} very conveniently. Firstly, for each fixed $k\in N$, $\beta_{\lambda, k}<0$, for all $\lambda>0$. Then by using Lemma \ref{Lem3.2}, \ref{Lem3.4}, $(i)$ of Lemma \ref{Lem3.3} and $\underline{(i)}$, $\underline{(iii)}$ of Theorem \ref{Thm3.2}, we can prove the assertion $(i)$ of the Theorem \ref{THm3.1}. Obviously, the assertion $(iii)$ of Theorem \ref{THm3.1} can be obtained from $\underline{(iii)}$ of Theorem \ref{Thm3.2}. Moreover, for $p\in[2+\frac{4}{N}, 4+\frac{4}{N})$, according to $(ii)$ of Lemma \ref{Lem3.3} and the Clark's theorem, it is easy to see (\ref{dnfj1.10}) has at least $k$ distinct radial solutions under $\mathbb{Y}_{\lambda}$.





\section*{Conflict of interest statement}
\addcontentsline{toc}{section}{Conflict of interest statement}

The authors declare no conflicts of interest.

\section*{Data availability statement}
\addcontentsline{toc}{section}{Data availability statement}

The data supporting this study's findings are available from the corresponding author upon reasonable request.


\end{document}